\documentclass{article}

\usepackage{hyperref}
\hypersetup{%
colorlinks=true,
linkcolor=blue,
urlcolor=black,
}

\usepackage{amsmath,amsthm,amssymb}
\usepackage{mathrsfs,bm}
\usepackage{mathtools,tablists,enumitem}

\usepackage{expl3}
\ExplSyntaxOn
\cs_new_eq:NN \Repeat \prg_replicate:nn
\ExplSyntaxOff

\newcommand{\longhookrightarrow}{\mathrel{\lhook\joinrel\relbar\joinrel\rightarrow}}

\mathtoolsset{showonlyrefs=true}

\newcommand{\jbk}[1]{\left\langle {#1} \right\rangle}

\newcommand{\df}{\mathrm{d}}

\newcommand{\nv}{\nu}

\newcommand{\p}{\partial}

\newcommand{\Rbb}{\mathbb{R}}

\newcommand{\Acal}{\mathcal{A}}

\newcommand{\Hcal}{\mathcal{H}}

\newcommand{\Kcal}{\mathcal{K}}
\newcommand{\Lcal}{\mathcal{L}}

\newcommand{\Scal}{\mathcal{S}}

\newcommand{\Ga}{\alpha}

\newcommand{\GO}{\Omega}

\newcommand{\freespace}{\vspace{4mm}}

\numberwithin{equation}{section}

\newtheorem{prop}{Proposition}[section]
\newtheorem{theo}[prop]{Theorem}
\newtheorem{coro}[prop]{Corollary}
\newtheorem{lemm}[prop]{Lemma}

\theoremstyle{definition}

\newtheorem*{claim*}{Claim}
\newtheorem*{exam*}{Example}
\newtheorem*{rema*}{Remark}

\usepackage{tikz}
\usetikzlibrary{math, calc, arrows.meta}

\author{Shota Fukushima\thanks{Corresponding author. Department of Mathematics and Institute of Applied Mathematics, Inha University, 100 Inha-ro, Incheon 22212, S. Korea. Email: \texttt{shota.fukushima.math@gmail.com}} \and Hyeonbae Kang\thanks{Department of Mathematics and Institute of Applied Mathematics, Inha University, 100 Inha-ro, Incheon 22212, S. Korea. Email: \texttt{hbkang@inha.ac.kr}} \and Yoshihisa Miyanishi\thanks{Department of Mathematical Sciences, Faculty of Science, Shinshu University, Asahi 3-1-1, Matsumoto 390-8621, Japan. Email: \texttt{miyanishi@shinshu-u.ac.jp}}}

\title{Decay rate of the eigenvalues of the Neumann-Poincar\'e operator\thanks{This work was supported by NRF of S. Korea grant No. 2022R1A2B5B01001445 and JSPS of Japan KAKENHI grants No. 21K13805 and No. 20K03655.}}

\date{\empty}

\begin{document}
\maketitle

\begin{abstract}
       If the boundary of a domain in three dimensions is smooth enough, then the decay rate of the eigenvalues of the Neumann-Poincar\'e operator is known and it is optimal. In this paper, we deal with domains with less regular boundaries and derive quantitative estimates for the decay rates of the Neumann-Poincar\'e eigenvalues in terms of the H\"older exponent of the boundary. Estimates in particular show that the less the regularity of the boundary is, the slower is the decay of the eigenvalues. We also prove that the similar estimates in two dimensions. The estimates are not only for less regular boundaries for which the decay rate was unknown, but also for regular ones for which the result of this paper makes a significant improvement over known results.
\end{abstract}



\tableofcontents

\section{Introduction}

Let $d\geq 1$ and $\Omega\subset \Rbb^{d+1}$ be a bounded domain with the Lipschitz boundary $\partial\Omega$. The Neumann-Poincar\'e operator (abbreviated to NP operator) is the boundary integral operator defined by
\begin{equation}
    \label{eq_np}
    \Kcal^* [f](x):=\frac{1}{\omega_{d+1}}\mathrm{p.v.}\int_{\partial\Omega} \frac{(x-y)\cdot \nv_x}{|x-y|^{d+1}}f(y)\, \df \sigma (y) \quad (x\in \partial\Omega).
\end{equation}
Here $\omega_{d+1}$ is the area of the unit sphere in $\Rbb^{d+1}$, $\nv_x$ is the outward unit normal vector at $x\in \partial\Omega$, and $\mathrm{p.v.}$ stands for the Cauchy principal value. The NP operator can be realized as a self-adjoint operator by introducing a proper inner product on $H^{-1/2}(\partial\Omega)$ \cite{KPS07}, and hence its spectrum on $H^{-1/2}(\partial\Omega)$ consists of essential spectrum and eigenvalues.

If $\partial\Omega$ is $C^{1,\alpha}$ for some $\alpha >0$, then the NP operator is compact on $H^{-1/2}(\partial\Omega)$ (see Corollary \ref{coro_NP_compact_Sobolev}). So it has eigenvalues of finite multiplicities accumulating to $0$.
We denote the $j$-th eigenvalue of $\Kcal^*$ by $\lambda_j=\lambda_j(\Kcal^*)$ which are enumerated in descending order, namely,
\[
    |\lambda_1 |\geq |\lambda_2|\geq \cdots (\to 0).
\]
In three-dimensional case ($d+1=3$), it is proved in \cite{Miyanishi22} that if $\partial\Omega$ is $C^{2, \alpha}$ for some $\alpha>0$, then $|\lambda_j|$ decays at the rate of $j^{-1/2}$. In fact, it is proved that $|\lambda_j|$ exhibits the asymptotic behavior
\[
    |\lambda_j|\sim Cj^{-1/2} \quad (j\to \infty)
\]
with some constant $C>0$ given explicitly in terms of the Willmore energy and the Euler characteristic of $\partial\Omega$. Moreover, the leading order terms of the asymptotic behaviors of positive and negative eigenvalues are obtained in \cite{MR19} when $\partial\Omega$ is smooth. However, the convergence rate of NP eigenvalues for $C^{1,\alpha}$ domains in three-dimensional space is unknown, and it is one of the purposes of this paper to investigate it. It is worth mentioning that the NP operator on the domains with $C^{0,1}$ boundaries, i.e., Lipschitz boundaries is not compact and may not have eigenvalues converging to $0$ as shown in \cite{HP1, HP2} for some special domains in three dimensions and in \cite{BZ,KLY, PP1,PP2} for intersecting discs and general two-dimensional curvilinear domains.

We also deal with the similar problem in two dimensions ($d+1=2$). For the two-dimensional case, it is proved in \cite{Miyanishi-Suzuki17} that, if $\partial\Omega$ is $C^k$ ($k\geq 2$), then the NP eigenvalues converges to $0$ at the rate of $j^{-k+3/2}$. More precisely, it is proved that $|\lambda_j|=o(j^{-k+3/2+\delta})$ as $j\to \infty$ for any $\delta>0$. In \cite{Jung-Lim20}, it is proved that, if $\Omega\subset \Rbb^2$ is a simply connected domain with $C^{k, \alpha}$-boundary for some $k\geq 1$ and $\alpha\in (0, 1]$ such that $k+\alpha>3/2$, then $|\lambda_j|=O(j^{-k-\alpha+3/2})$ as $j\to \infty$. In particular, in the case of $k=1$ and $\alpha>1/2$, the decay estimate $|\lambda_j|=O(j^{-\alpha+1/2})$ holds. So, the case of $C^{1,\alpha}$ with $0<\alpha<1/2$ is missing. Moreover, this result shows that if $k=1$ and $\alpha=1/2$, then the decay rate is $j^0$, which leaves no room for the case of $C^{1, \alpha}$ with $0< \alpha < 1/2$ even though eigenvalues still decay to $0$ in such cases. So, we consider $C^{k,\alpha}$ domains and obtain decay estimates for $k \ge 1$ and $\alpha \in (0,1]$. It turns out that previous results for the case $k+\alpha \ge 3/2$ can be significantly improved. For instance, it can be proved that if $k=1$ and $\alpha=1/2$, then the critical decay is $j^{-1/2}$.

The main result of this paper in three dimensions is the following:

\begin{theo}\label{theo_eigenvalue_decay}
    Let $\Omega$ be a bounded domain in $\Rbb^{3}$ with $C^{1, \alpha}$ boundary $\partial\Omega$ for some $\alpha \in (0,1]$. We have
    \begin{equation}
        \label{eq_eigenvalue_decay}
        |\lambda_j (\Kcal^*)|=o(j^{-\alpha/2+\delta}) \quad (j\to \infty)
    \end{equation}
    for any $\delta>0$.
\end{theo}

In two dimensions, the main result is the following:

\begin{theo}
    \label{theo_eigenvalue_decay_2d_smoother}
    Let $\Omega$ be a bounded domain in $\Rbb^{2}$ with $C^{k, \alpha}$ boundary $\partial\Omega$ for some positive integer $k$ and $\alpha \in (0,1]$.  We have
    \begin{equation}\label{700}
        |\lambda_j (\Kcal^*)|=o(j^{-k+1 -\alpha+\delta}) \quad (j\to \infty)
    \end{equation}
    for all $\delta>0$.
\end{theo}

Theorem \ref{theo_eigenvalue_decay} and \ref{theo_eigenvalue_decay_2d_smoother} show that the critical decay exponent is \begin{equation}\label{1000}
q=
\begin{cases}
\alpha/2 \quad &\text{if } d=2, \\
k-1+\alpha \quad &\text{if } d=1
\end{cases}
\end{equation}
in the sense that $|\lambda_j (\Kcal^*)|=o(j^{-q+\delta})$ for any $\delta>0$. As Figure \ref{fg_pa_3d} shows, the line of critical exponent connects $0$ (when $\alpha =0$) and $1/2$ (when $\alpha=1$). The critical exponent $1/2$ is proved in \cite{Miyanishi22} as mentioned at the beginning of Introduction. Figure \ref{fg_pa_2d} compares the critical exponent of this paper and that obtained in \cite{Jung-Lim20, Miyanishi-Suzuki17}.

\begin{figure}[htb]
    \begin{tabular}[htb]{cc}
        \begin{minipage}{0.45\hsize}
            \begin{tikzpicture}[scale=0.6]
                \node (C3) at (4, 3) {};
                \node (C2) at (2, 1) {};
                \node (C5/2) at (5, 4) {};
                \node (C7/4) at (1.5, 0.5) {};
                \node (C7/4new) at (1.2, 1.2) {};
                \node (C5/2expected) at (4.5, 4.5) {};
                \node[font=\small] (MS) at (5, 2) {\cite{Miyanishi-Suzuki17}};
                \node[font=\small] (JL) at (6, 1) {\cite{Jung-Lim20}};
                \node[font=\small] (new) at (1.3, 3.5) {this paper};
                \node[font=\small] (origin) at (0, 0) [below left] {$O$};
                \node[font=\small] (alpha) at (7, 0) [below] {$k+\alpha-1$};
                \node[font=\small] (beta) at (0, 5) [left] {$q$};
                \node[font=\small] (x1) at (2, 0) [below] {$1$};
                \node[font=\small] (x1/2) at (1, 0) [below] {$1/2$};
                \node[font=\small] (y1/2) at (0, 1) [left] {$\displaystyle\frac{1}{2}$};
                \node[font=\small] (x2) at (4, 0) [below] {$2$};
                \node[font=\small] (y1) at (0, 2) [left] {$1$};
                \node[font=\small] (y3/2) at (0, 3) [left] {$\displaystyle\frac{3}{2}$};
                \node[font=\small] (y2) at (0, 4) [left] {$2$};
                \draw[dotted] (1, 0)--(6, 5);
                \draw[thick] (0, 0)--(5, 5);
                \draw[dashed] (0, 3)--(C3);
                \draw[dashed] (2, 0)--(2, 2)--(0, 2);
                \draw[dashed] (0, 1)--(2, 1);
                \draw[dashed] (4, 0)--(4, 4)--(0, 4);
                \draw[->] (0, 0)--(7, 0);
                \draw[->] (0, 0)--(0, 5);
                \filldraw[fill=white] (4, 3) circle [radius=0.08];
                \filldraw[fill=white] (0, 0) circle [radius=0.08];
                \filldraw[fill=white] (2, 1) circle [radius=0.08];
                \draw[->, >=stealth] (MS)--(C3);
                \draw[->, >=stealth] (MS)--(C2);
                \draw[->, >=stealth] (JL)--(C5/2);
                \draw[->, >=stealth] (JL)--(C7/4);
                \draw[->, >=stealth] (new)--(C7/4new);
            \end{tikzpicture}
            \caption{Critical decay exponent for $d=1$.}
            \label{fg_pa_2d}
        \end{minipage} &
        \begin{minipage}{0.45\hsize}
            \begin{tikzpicture}[scale=0.6]
                \node (C3) at (6, 2) {};
                \node (C2) at (4, 2) {};
                \node (C3/2) at (2, 1) {};
                \node[font=\small] (M) at (6, 4) {\cite{Miyanishi22}};
                \node[font=\small] (new) at (2, 3) {{\small this paper}};
                \node[font=\small] (origin) at (0, 0) [below left] {$O$};
                \node[font=\small] (alpha) at (7, 0) [below] {$k+\alpha-1$};
                \node[font=\small] (beta) at (0, 5) [left] {$q$};
                \node[font=\small] (x1) at (4, 0) [below] {$1$};
                \node[font=\small] (y1/2) at (0, 2) [left] {$\displaystyle\frac{1}{2}$};
                \draw[thick] (0, 0)--(4, 2)--(7, 2);
                \draw[dashed] (4, 0)--(C2)--(0, 2);
                \draw[->] (0, 0)--(7, 0);
                \draw[->] (0, 0)--(0, 5);
                \filldraw[fill=white] (0, 0) circle [radius=0.08];
                \fill (C2) circle [radius=0.08];
                \draw[->, >=stealth] (M)--(C3);
                \draw[->, >=stealth] (new)--(C3/2);
            \end{tikzpicture}
            \caption{Critical decay exponent for $d=2$.}
            \label{fg_pa_3d}
        \end{minipage}
    \end{tabular}
\end{figure}

Theorem \ref{theo_eigenvalue_decay} and \ref{theo_eigenvalue_decay_2d_smoother} are proved using a result of \cite{Delgado-Ruzhansky14} in a crucial way. The result is about a necessary condition of a compact integral operator on a Hilbert space to belong to a Schatten class. Using this result we show that a composition of the NP operators belongs to a certain Schatten class. We then use the Weyl's lemma, which asserts the $\ell^p$-norm of eigenvalues is less than that of singular values, to derive decay rates of eigenvalues. The necessary condition in \cite{Delgado-Ruzhansky14} is given in terms of the Sobolev norm of the integral kernel of the operator. It says that the higher differentiability of the integral kernel implies $p$-Schatten class for the operator for smaller $p$, and in turn implies a faster decay rate of eigenvalues. This explains the difference between decay rates in two and three dimensions as depicted in Figure \ref{fg_pa_2d} and \ref{fg_pa_3d}. If $\Omega \subset \Rbb^2$ and $\partial\Omega$ is $C^{k,\alpha}$, the integral kernel of the NP operator gains differentiability and eigenvalues decay indefinitely faster as $k$ or $\alpha$ increases as is proved in section \ref{theorem1.2} (if $\partial\Omega$ is real analytic, then the decay of eigenvalue is exponentially fast \cite{AKM18}). However, the integral kernel in three dimensions gains differentiability
only when $k=1$ and $\alpha$ increases; if $k \ge 2$, it does not and eigenvalues decay at the rate of $j^{-1/2}$ no matter how large $k$ is.

We do not know whether or not the estimates in Theorem \ref{theo_eigenvalue_decay} and \ref{theo_eigenvalue_decay_2d_smoother} are optimal. In this regard, it is interesting to investigate if the estimate $|\lambda_j|=O(j^{-q})$ holds, where $q$ is the critical exponent given in \eqref{1000}.

This paper is organized as follows. In section \ref{theorem1.1}, we give a proof of Theorem \ref{theo_eigenvalue_decay}. Since the proof works for the case $d=1$ as well, we use $d=1,2$ there. So it is an alternative proof for Theorem \ref{theo_eigenvalue_decay_2d_smoother} when $k=1$. The proof of Theorem \ref{theo_eigenvalue_decay_2d_smoother} given in section \ref{theorem1.2} uses heavily the differentiability properties of the kernel of the NP operator. Appendix is to prove regularity properties of the NP operator on Sobolev spaces. As a consequence of the regularity properties, compactness of the NP operator on $H^{-1/2}(\partial\Omega)$ follows. It is also proved that eigenfunctions in $H^{-1/2}(\partial\Omega)$ corresponding to nonzero eigenvalues belong to $L^2(\partial\Omega)$.

\section{Proof of Theorem \ref{theo_eigenvalue_decay}}\label{theorem1.1}

\subsection{Singularity estimates of the integral kernel}

For the proof of Theorem \ref{theo_eigenvalue_decay}, we begin with an estimate of an integral which appears as the integral kernel of the composition of operators. Here and afterwards, we use the conventional notation $A \lesssim B$ to imply that there is a positive constant $C$ such that $A \le CB$.

\begin{lemm}
    \label{lemm_convolution_sing}
    Let $\Omega\subset \Rbb^{d+1}$ be a bounded domain with the Lipschitz boundary and $\alpha, \beta\in (0, d]$ satisfy $\alpha+\beta\leq d$. Then the following estimates hold for all $x, y\in \partial\Omega$ with $x\neq y$:
    \begin{equation}\label{eq_convolution_sing}
        \int_{\partial\Omega} \frac{1}{|x-z|^{d-\alpha}}\frac{1}{|z-y|^{d-\beta}}\, \df \sigma (z)\lesssim
        \begin{cases}
            |x-y|^{-d+\alpha+\beta} & \text{if } \alpha+\beta<d, \\
            \left|\log |x-y|\right| & \text{if } \alpha+\beta=d.
        \end{cases}
    \end{equation}
\end{lemm}

\begin{proof}
    We set $r=|x-y|$ and decompose the integral in \eqref{eq_convolution_sing} as
    \begin{align*}
        \int_{\partial\Omega} \frac{1}{|x-z|^{d-\alpha}}\frac{1}{|z-y|^{d-\beta}}\, \df \sigma (z) &= \int_{|x-z|<r/2} + \int_{r/2 \leq |x-z| < 2r} + \int_{|x-z|\geq 2r} \\
        &=:I_1+I_2+I_3.
    \end{align*}
    If $|x-z|<r/2$, then $|z-y|\geq |x-y|-|x-z|>r/2$. So we have
    \begin{align*}
        I_1&\lesssim \frac{1}{r^{d-\beta}}\int_{|x-z|<r/2} \frac{\df \sigma (z)}{|x-z|^{d-\alpha}}
        \lesssim r^{-d+\beta}\int_0^{r} \frac{\df \rho}{\rho^{1-\alpha}} \lesssim |x-y|^{-d+\alpha+\beta}.
    \end{align*}
    If $r/2 \leq |x-z| < 2r$, then $|z-y| \le 3r$. So we have
    \[
        I_2 \lesssim \frac{1}{r^{d-\alpha}}\int_{|y-z|<3r} \frac{\df \sigma (z)}{|y-z|^{d-\beta}}
        \lesssim |x-y|^{-d+\alpha+\beta}.
    \]
If $|x-z|\geq 2r$, then $\frac{1}{2} |z-y| \le |x-z| \le 2|z-y|$. So we have
    \begin{align*}
        I_3 &\lesssim \int_{|x-z|\geq 2r} \frac{1}{|x-z|^{2d-\alpha-\beta}}\, \df \sigma (z)
        \lesssim
        \begin{cases}
            r^{-d+\alpha+\beta} & \text{if } \alpha+\beta<d, \\
            \left|\log r\right| & \text{if } \alpha+\beta=d.
        \end{cases}
    \end{align*}
This completes the proof.
\end{proof}

For a positive integer $n$, we define a linear operator $\Lcal_n$ by
\begin{equation}\label{300}
    \Lcal_n:=
    \begin{cases}
        \Kcal^*(\Kcal\Kcal^*)^{(n-1)/2} & \text{if } n \text{ is odd}, \\
        (\Kcal\Kcal^*)^{n/2} & \text{if } n \text{ is even}
    \end{cases}
\end{equation}
and denote the integral kernel of $\Lcal_n$ by $L_n(x, y)$:
\begin{equation}
    \label{eq_np_integral_kernel}
    \Lcal_n [f](x)=\int_{\partial\Omega} L_n (x, y)f(y)\, \df \sigma (y).
\end{equation}
Note that $\Lcal_1=\Kcal^*$ and thus $L_1(x, y)$ is the integral kernel of the NP operator $\Kcal^*$. We emphasize that if $\partial \Omega$ is $C^{1,\alpha}$, then
\begin{equation}\label{500}
|L_1(x, y)| \lesssim \frac{1}{|x-y|^{d-\alpha}}.
\end{equation}

We derive estimates for $L_n (x, y)-L_n (x, y^\prime)$, which will be used later. We begin with the case of $n=1$.

\begin{lemm}
    \label{lemm_L1_kernel_estimate}
        There exists a constant $C>0$ such that, if $2|y-y^\prime|<|x-y|$, then we have
        \begin{equation}\label{100}
            |L_1(x, y)-L_1(x, y^\prime)|\leq \frac{C|y-y^\prime|}{|x-y|^{d+1-\alpha}}.
        \end{equation}
\end{lemm}

\begin{proof}
    For simplicity of notation, we set $t=|y-y^\prime|$ and $s=|x-y|$.
    We have
    \begin{align*}
        &|L_1(x, y)-L_1(x, y^\prime)| \\
        &\lesssim \frac{|\nv_x \cdot (y-y^\prime)|}{|x-y|^{d+1}}
        +\left|\frac{1}{|x-y|^{d+1}}-\frac{1}{|x-y^\prime|^{d+1}}\right| |\nv_x \cdot (x-y^\prime)|=:A_1+A_2.
    \end{align*}

    Since $\p\GO$ is $C^{1,\alpha}$, we have
    \begin{align*}
        |\nv_x \cdot (y-y^\prime)|\leq |(\nv_x-\nv_y)\cdot (y-y^\prime)|+|\nv_y \cdot (y-y^\prime)|\lesssim s^\alpha t+t^{1+\alpha}
        \lesssim s^\alpha t,
    \end{align*}
    from which we infer that
    \[
        A_1 \lesssim \frac{t}{s^{d+1-\alpha}}.
    \]

    Since $2 t<s$, we have
    \[
    |x-y^\prime|\geq |x-y|-|y-y^\prime|\geq s/2
    \]
    and
    \[
    |x-y^\prime|\leq |x-y|+|y-y^\prime|\leq 3s/2.
    \]
    It then follows that
    \begin{align*}
        \left|\frac{1}{|x-y|^{d+1}}-\frac{1}{|x-y^\prime|^{d+1}}\right|
        &=\frac{||x-y^\prime|^{d+1}-|x-y|^{d+1}|}{|x-y^\prime|^{d+1}|x-y|^{d+1}} \\
        &\lesssim \frac{||x-y|-|x-y^\prime||}{s^{2d+2}}\sum_{j=1}^{d+1} |x-y^\prime|^{j-1}|x-y|^{d+1-j} \\
        &\lesssim \frac{|y-y^\prime|}{s^{2d+2}}\times s^d
        = \frac{t}{s^{d+2}}.
    \end{align*}
    Moreover, we have
    \begin{align*}
        |\nv_x \cdot (x-y^\prime)|\leq |x-y^\prime|^{1+\alpha}\lesssim s^{1+\alpha},
    \end{align*}
    which together with the above estimate leads us to
    \[
        A_2 \lesssim \frac{t}{s^{d+1-\alpha}}.
    \]
    So, \eqref{100} follows.
\end{proof}

We now deal with the case when $n \ge 2$.
\begin{lemm}
    \label{lemm_Ln_kernel_estimate}
    Assume $\alpha\leq d/2$ and let $n$ be an integer such that $n\leq d/\alpha$.
    Then there exists a constant $C>0$ such that, if $4|y-y^\prime|< |x-y|$, then we have
    \begin{equation}\label{200}
        |L_n(x, y)-L_n(x, y^\prime)|\leq \frac{C|y-y^\prime|^\alpha}{|x-y|^{d-(n-1)\alpha}}.
    \end{equation}
\end{lemm}

\begin{proof}
    Since the case $n=1$ was already proved in Lemma \ref{lemm_L1_kernel_estimate}, we assume $n \ge 2$. We denote by $L^\prime_n(x, y)$ the integral kernel of $(\Kcal^*\Kcal)^{(n-1)/2}$ if $n$ is odd and that of $(\Kcal\Kcal^*)^{n/2-1}\Kcal$ if $n$ is even. Note that the following relation holds: 
    \begin{equation}\label{2000}
        L_n(x, y)-L_n(x, y^\prime) 
        = \int_{\partial\Omega} L^\prime_n(x, z)
        (L_1(z, y)-L_1(z, y^\prime))\, \df \sigma (z).
    \end{equation}

    Since $(n-1)\alpha<d$, we may apply Lemma \ref{lemm_convolution_sing} repeatedly to infer that
    $$
    |L^\prime_n(x, z)|\lesssim \frac{1}{|x-z|^{d-(n-1)\alpha}}
    $$
    for all $x, z\in \partial\Omega$ with $x\neq z$. It then follows from \eqref{2000} that 
    \begin{equation}
        \label{eq_Ln_Lpn}
        |L_n(x, y)-L_n(x, y^\prime)|\lesssim \int_{\partial\Omega} \frac{|L_1(z, y)-L_1(z, y^\prime)|}{|x-z|^{d-(n-1)\alpha}}\, \df \sigma (z).
    \end{equation}

    For simplicity of notation, we put $t=|y-y^\prime|$ and $s=|x-y|$. We decompose the integral in \eqref{eq_Ln_Lpn} into two parts:
    \[
        \int_{\partial\Omega} \frac{|L_1(z, y)-L_1(z, y^\prime)|}{|x-z|^{d-(n-1)\alpha}}\, \df \sigma (z)= \int_{|z-y|\leq 2t} + \int_{|z-y|> 2t}=:I+J.
    \]

    If $|z-y|\leq 2t$, then $|x-z|> |x-y| - |y-z| >s/2$ since $4t< s$, and $I$ can be estimated as follows:
    \begin{align*}
        I &\lesssim \int_{|z-y|\leq 2t} \frac{|L_1(z, y)|+|L_1(z, y^\prime)|}{|x-z|^{d-(n-1)\alpha}}\, \df \sigma (z) \\
        &\lesssim \int_{|z-y|\leq 2t} \frac{1}{|x-z|^{d-(n-1)\alpha}}\left(\frac{1}{|z-y|^{d-\alpha}}+\frac{1}{|z-y^\prime|^{d-\alpha}}\right)\, \df \sigma (z) \\
        &\lesssim \frac{1}{s^{d-(n-1)\alpha}}\left(\int_{|z-y|\leq 2t} \frac{1}{|z-y|^{d-\alpha}}\, \df \sigma (z)+\int_{|z-y^\prime|\leq 3t}\frac{1}{|z-y^\prime|^{d-\alpha}}\, \df \sigma (z)\right) \\
        &\lesssim \frac{t^\alpha}{s^{d-(n-1)\alpha}}.
    \end{align*}

    By Lemma \ref{lemm_L1_kernel_estimate}, we have
    \[
        J\lesssim t\int_{|z-y|>2t} \frac{\df \sigma (z)}{|z-y|^{d+1-\alpha}|x-z|^{d-(n-1)\alpha}}.
    \]
    We decompose the integral in the right hand side into three parts:
    \begin{align*}
        &\int_{|z-y|>2t} \frac{\df \sigma (z)}{|z-y|^{d+1-\alpha}|x-z|^{d-(n-1)\alpha}} \\
        & = \int_{2t<|z-y|\leq s/2} + \int_{s/2<|z-y|\leq 2s} + \int_{2s<|z-y|} =: J_1 + J_2 + J_3.
    \end{align*}

    If $|z-y|\leq s/2$, then $|x-z|\geq |x-y|-|y-z|\geq s/2$, and hence the integral $J_1$ is estimated as
    \begin{align*}
        J_1\lesssim \frac{1}{s^{d-(n-1)\alpha}}\int_{2t<|z-y|\leq s/2} \frac{\df \sigma (z)}{|z-y|^{d+1-\alpha}}
        \lesssim \frac{1}{s^{d-(n-1)\alpha}t^{1-\alpha}}.
    \end{align*}
    If $|z-y|\leq 2s$, then $|x-z|\leq |x-y|+|y-z|\leq 3s$, and hence $J_2$ is estimated as
    \begin{align*}
        J_2\lesssim \frac{1}{s^{d+1-\alpha}}\int_{|x-z|\leq 3s} \frac{\df \sigma (z)}{|x-z|^{d-(n-1)\alpha}}
        \lesssim \frac{1}{s^{d+1-n\alpha}}\lesssim \frac{1}{s^{d-(n-1)\alpha}t^{1-\alpha}},
    \end{align*}
    where the last inequality holds since $4t \leq s$. If $|z-y|>2s$, then $|x-z|\geq |z-y|-|x-y|>|z-y|/2$, and hence $J_3$ is estimated as
    \begin{align*}
        J_3\lesssim \int_{2s<|z-y|}\frac{\df \sigma (z)}{|z-y|^{2d+1-n\alpha}} \lesssim \frac{1}{s^{d+1-n\alpha}}\lesssim \frac{1}{s^{d-(n-1)\alpha}t^{1-\alpha}}.
    \end{align*}
    Thus we obtain
    \[
        J\leq \frac{Ct}{s^{d-(n-1)\alpha} t^{1-\alpha}}\leq \frac{Ct^\alpha}{s^{d-(n-1)\alpha}},
    \]
    and the proof is complete.
\end{proof}

\subsection{Estimates of the kernel in Sobolev spaces}

Let $H^{\mu, \nu} (\partial\Omega \times \partial\Omega)$ be the Sobolev space on $\partial\Omega\times \partial\Omega$ of order $\mu\geq 0$ with respect to the first variable and order $\nu\geq 0$ with respect to the second variable (for more detail, see \cite{Delgado-Ruzhansky14}). We only employ the case $\mu=0$ in this paper. In this case, the linear mapping
\[
    A\in H^{0, \nu}(\partial\Omega \times \partial\Omega) \longmapsto (x \mapsto A(x, \cdot))\in L^2(\partial\Omega, H^\nu (\partial\Omega))
\]
gives a natural isomorphism $H^{0, \nu}(\partial\Omega \times \partial\Omega)\simeq L^2 (\partial\Omega, H^\nu (\partial\Omega))$. Here $H^\nu (\partial\Omega)$ is the Sobolev space on $\partial\Omega$ of order $\nu$. It is known that, if $k\geq 0$ is an integer and $0<\alpha <1$, then the Sobolev norm on $H^{k+\alpha}(\partial\Omega)$ is equivalent to the norm
\begin{equation}
    \label{eq_sobolev}
    \left( \sum_{j=0}^k\| \partial_{\tau}^j f\|_{L^2(\partial\Omega)}^2+\int_{\partial\Omega \times \partial\Omega}\frac{|\partial_{\tau}^k f(x)-\partial_{\tau}^k f(y)|^2}{|x-y|^{2\alpha+d}}\, \df \sigma (x)\df \sigma (y)\right)^{1/2}
\end{equation}
where $\partial_\tau$ is the derivative in the tangential direction (see \cite{Gilbarg-Trudinger01}). In this paper, we denote the norm \eqref{eq_sobolev} by $\|f\|_{H^{k+\alpha}}$. Thus, the norm
\begin{equation}
    \label{eq_sobolev_double}
    \begin{aligned}
        &\left(\int_{\partial\Omega} \| A(x, \cdot)\|_{H^{k+\alpha}}^2 \, \df \sigma (x)\right)^{1/2} \\
    &=\biggl( \sum_{j=0}^k \| \partial_{\tau_y}^j A\|_{L^2(\partial\Omega \times \partial\Omega)}^2 \\
    &\quad +\int_{(\partial\Omega)^3} \frac{|\partial_{\tau_y}^k A(x, y)-\partial_{\tau_y}^k A(x, y^\prime)|^2}{|y-y^\prime |^{2\alpha+d}} \, \df \sigma (x) \df \sigma (y) \df \sigma (y^\prime)\biggr)^{1/2}
    \end{aligned}
\end{equation}
is equivalent to the norm on $H^{0, k+\alpha} (\partial\Omega \times \partial \Omega)$.

\begin{lemm}\label{lemm_kernel_sobolev}
    Let $0<\alpha\le 1$ and $\partial\Omega$ is $C^{1, \alpha}$. Assume  and let $n$ be the smallest integer larger than $d/(2\alpha)$. Then $L_n \in H^{0, \nu}(\partial\Omega \times \partial\Omega)$ for all $\nu \in (0, (2n\alpha-d)/2)$, where $L_n$ is the integral kernel of the operator $\Lcal_n$ defined in \eqref{300}.
\end{lemm}

\begin{proof}
    Assume $0<\nu < (2n\alpha-d)/2$. We shall show
    \begin{equation}
        \label{eq_infty_sobolev}
        \sup_{x\in \partial\Omega} \| L_n (x, \cdot)\|_{H^\nu}<\infty,
    \end{equation}
    which immediately implies $L_n\in H^{0, \nu}(\partial\Omega \times \partial\Omega)$ by the representation \eqref{eq_sobolev_double} of the Sobolev norm on $H^{0, \nu}(\partial\Omega \times \partial\Omega)$.

    We first deal with the case when $\alpha\leq d/2$. Since the integral kernels of both $\Kcal$ and $\Kcal^*$ are $O(|x-y|^{-d+\alpha})$ as $|x-y|\to 0$ and $n\alpha\leq d/2+\alpha\leq d$, one can show by applying Lemma \ref{lemm_convolution_sing} repeatedly that the following estimates hold:
    \begin{equation}\label{400}
        |L_n (x, y)|=
        \begin{cases}
            O(|x-y|^{-d+n\alpha}) & \text{if } n\alpha<d, \\
            O(|\log |x-y||) & \text{if } n\alpha=d.
        \end{cases}
    \end{equation}
    Since $n\alpha \ge d/2$, we have
    $$
    \sup_{x\in \partial\Omega}\| L_n(x, \cdot)\|_{L^2(\partial\Omega)}<\infty
    $$
    in both cases.

    We now show that
    \begin{equation}
        \label{eq_kernel_sobolev_seminorm}
        \sup_{x\in \partial\Omega}\int_{\partial\Omega \times \partial\Omega} \frac{|L_n (x, y)-L_n (x, y^\prime)|^2}{|y-y^\prime |^{2\nu+d}} \, \df \sigma (y) \df \sigma (y^\prime)<\infty.
    \end{equation}
    We decompose $\partial\Omega \times \partial\Omega$ into two regions $4|y-y^\prime|< |x-y|$ and $4|y-y^\prime| \ge |x-y|$ and make estimates there separately.

    In the region $4|y-y^\prime|< |x-y|$, we apply Lemma \ref{lemm_Ln_kernel_estimate} to obtain
    \begin{align*}
        I_1 & :=\int_{4|y-y^\prime|< |x-y|} \frac{|L_n (x, y)-L_n (x, y^\prime)|^2}{|y-y^\prime |^{2\nu+d}} \, \df \sigma (y) \df \sigma (y^\prime) \\
        &\lesssim \int_{\partial\Omega} \frac{\df \sigma (y)}{|x-y|^{2(d-(n-1)\alpha)}}\int_{4|y-y^\prime|<|x-y|} \frac{\df \sigma (y^\prime)}{|y-y^\prime|^{2\nu+d-2\alpha}} .
    \end{align*}
    Since $2\nu+d-2\alpha < 2(n-1) \alpha \le d$, we have
    \begin{align*}
    \int_{4|y-y^\prime|<|x-y|} \frac{\df \sigma (y^\prime)}{|y-y^\prime|^{2\nu+d-2\alpha}} \lesssim \frac{1}{|x-y|^{2\nu-2\alpha}},
    \end{align*}
    and hence
    \begin{align*}
        I_1 \lesssim \int_{\partial\Omega} \frac{\df \sigma (y)}{|x-y|^{2(d-n\alpha)+2\nu}}\lesssim 1,
    \end{align*}
    where the second inequality holds since $\nu<(2n\alpha-d)/2$.

    In the region $|x-y|\leq 4|y-y^\prime|$, we shall estimate
    \[
        I_{2,1}:=\int_{|x-y|\leq 4|y-y^\prime|} \frac{|L_n (x, y)|^2}{|y-y^\prime |^{2\nu+d}} \, \df \sigma (y) \df \sigma (y^\prime)
    \]
    and
    \[
        I_{2.2}:=\int_{|x-y|\leq 4|y-y^\prime|} \frac{|L_n (x, y^\prime)|^2}{|y-y^\prime |^{2\nu+d}} \, \df \sigma (y) \df \sigma (y^\prime)
    \]
    separately. We first note that $n\alpha\le d$. If $n\alpha<d$, then we use the first inequality in \eqref{400} to have
    \begin{align*}
        I_{2,1}
        &\lesssim \int_{\partial\Omega} \frac{\df \sigma (y)}{|x-y|^{2(d-n\alpha)}}\int_{|x-y|\leq 4|y-y^\prime|} \frac{\df \sigma (y^\prime)}{|y-y^\prime|^{2\nu+d}} \\
        &\lesssim \int_{\partial\Omega} \frac{\df \sigma (y)}{|x-y|^{2(d-n\alpha)+2\nu}}\lesssim 1.
    \end{align*}
    If $n\alpha=d$, then we use the second inequality in \eqref{400} to have
    \begin{align*}
        I_{2,1}
        &\lesssim \int_{\partial\Omega} \df \sigma (y)\, |\log |x-y||\int_{|x-y|\leq 4|y-y^\prime|} \frac{\df \sigma (y^\prime)}{|y-y^\prime|^{2\nu+d}} \\
        &\lesssim \int_{\partial\Omega} \frac{|\log |x-y||}{|x-y|^{2\nu}}\,\df \sigma (y)\lesssim 1,
    \end{align*}
    where the last inequality holds since $\nu <(2n\alpha-d)/2=d/2$.

    Since $|x-y^\prime|\leq |x-y|+|y-y^\prime|\leq 5|y-y^\prime|$ if $|x-y|\leq 4|y-y^\prime|$, it can be proved similarly that $I_{2,2} \lesssim 1$. This completes the proof when $\alpha\leq d/2$.

    Suppose now that $\alpha > d/2$. Then $d=1$, $1/2 < \alpha \le 1$, and $n=1$. We also have $\nu < 1$ in this case.
    By Lemma \ref{lemm_L1_kernel_estimate}, we have
    \begin{align*}
        &\int_{2|y-y^\prime|<|x-y|} \frac{|L_1(x, y)-L_1(x, y^\prime)|^2}{|y-y^\prime|^{2\nu+1}}\, \df \sigma (y)\df \sigma (y^\prime) \\
        &\lesssim \int_{\partial\Omega} \frac{\df \sigma (y)}{|x-y|^{4-2\alpha}}\int_{2|y-y^\prime|<|x-y|} \frac{\df \sigma (y^\prime)}{|y-y^\prime|^{2\nu-1}} \\
        &\lesssim \int_{\partial\Omega} \frac{\df \sigma (y)}{|x-y|^{4-2\alpha+2\nu-2}} \lesssim 1,
    \end{align*}
    where the second inequality holds since $\nu<1$ and the third one since $\nu<(2\alpha-1)/2$.

    In the region $2|y-y^\prime|\geq |x-y|$, we use \eqref{500} to have
    \begin{align*}
        &\int_{2|y-y^\prime|\geq |x-y|} \frac{|L_1(x, y)|^2}{|y-y^\prime|^{2\nu+1}}\, \df \sigma (y)\df \sigma (y^\prime) \\
        &\lesssim \int_{\partial\Omega} \frac{\df \sigma (y)}{|x-y|^{2-2\alpha}}\int_{2|y-y^\prime|\geq |x-y|} \frac{\df \sigma (y^\prime)}{|y-y^\prime|^{2\nu+1}} \\
        &\lesssim \int_{\partial\Omega} \frac{\df \sigma (y)}{|x-y|^{2-2\alpha+2\nu}}\lesssim 1.
    \end{align*}
    Since $|x-y^\prime|\leq |x-y|+|y-y^\prime|\leq 3|y-y^\prime|$ if $2|y-y^\prime|\geq |x-y|$, one can show in the same way that
    \begin{align*}
        &\int_{2|y-y^\prime|\geq |x-y|} \frac{|L_1(x, y^\prime)|^2}{|y-y^\prime|^{2\nu+1}}\, \df \sigma (y)\df \sigma (y^\prime) \lesssim 1.
    \end{align*}
    It then follows that
    \[
    \int_{2|y-y^\prime|\geq |x-y|} \frac{|L_1(x, y)-L_1(x, y^\prime)|^2}{|y-y^\prime|^{2\nu+1}}\, \df \sigma (y)\df \sigma (y^\prime) \lesssim 1.
    \]
    This completes the proof.
\end{proof}

\subsection{Schatten class and decay rates}

We denote for $p >0$ the $p$-th Schatten class associated with a Hilbert space $\Hcal$ by $\Scal^p (\Hcal)$. It is the class of compact operators $\Acal$ on $\Hcal$ such that $\sum_{j=1}^\infty s_j(\Acal)^p <\infty$, where $s_j(\Acal)$ is
the $j$-th singular value of $\Acal$ on $\Hcal$ enumerated in descending order, namely,
\[
    s_1 (\Acal)\geq s_2(\Acal)\geq \cdots (\to 0).
\]
We denote eigenvalues of $\Acal$ by $\lambda_j (\Acal)$, which are enumerated in the order descending in the absolute values, namely,
\[
|\lambda_1 (\Acal)| \geq |\lambda_2 (\Acal)| \geq \cdots (\to 0).
\]
The following lemma shows the decay rate of eigenvalues of compact operators in the Schatten class.

\begin{lemm}
    \label{lemm_schatten_decay}
    Let $p> 0$. If a compact operator $\Acal$ on a Hilbert space $\Hcal$ belongs to the Schatten class $\Scal^p (\Hcal)$, then we have
    \[
        |\lambda_j (\Acal)|=o(j^{-1/p})
    \]
    as $j\to \infty$.
\end{lemm}

\begin{proof}
    By the Weyl inequality, we have
    \[
        \sum_{j=1}^\infty |\lambda_j (\Acal)|^p \leq \sum_{j=1}^\infty s_j (\Acal)^p<\infty
    \]
    (see \cite[p.93]{Gohberg-Krein69} for example). Let $N$ be an arbitrary integer. Since $|\lambda_j(\Acal)|\geq |\lambda_{j+1}(\Acal)|$ for all $j\geq 1$, we have
    \[
        (j-N) |\lambda_j (\Acal)|^p\leq \sum_{k=N+1}^j |\lambda_j (\Acal)|^p\leq \sum_{k=N+1}^\infty |\lambda_k (\Acal)|^p
    \]
    for any $j>N$. Since $|\lambda_j(\Acal)|\to 0$ as $j\to \infty$, we have
    \[
        \limsup_{j\to \infty} j|\lambda_j (\Acal)|^p=\limsup_{j\to\infty}((j-N)|\lambda_j(\Acal)|^p)\leq \sum_{k=N+1}^\infty |\lambda_k(\Acal)|^p.
    \]
    By sending $N\to \infty$, we obtain $\limsup_{j\to \infty} j|\lambda_j(\Acal)|^p=0$.
\end{proof}

To prove Theorem \ref{theo_eigenvalue_decay}, we use a result of \cite{Delgado-Ruzhansky14}, which states that if an integral kernel belongs to a certain Sobolev space, then the associated integral operator belongs to a certain Schatten class.

\begin{theo}[{\cite[Theorem 3.6]{Delgado-Ruzhansky14}}]
    \label{theo_Delgado_Ruzhansky}
    Let $M$ be a $d$-dimensional compact oriented manifold with some nowhere-vanishing smooth density fixed (for example, equip $M$ with a Riemannian metric and consider the associated volume form). Let $\Acal$ be the integral operator
    \[
        \Acal [f] (x)=\int_M A(x, y)f(y)\, \df y
    \]
    and let $\mu, \nu\geq 0$. If $A(x, y)\in H^{\mu, \nu}(M\times M)$, then we have $\Acal\in \Scal^p (L^2(M))$ for all
    \[
        p>\frac{2d}{d+2(\mu+\nu)}.
    \]
\end{theo}

By Theorem \ref{theo_Delgado_Ruzhansky}, we can prove that the NP operator $\Kcal^*$ belongs to certain Schatten class depending on the smoothness of $\partial \Omega$.
In what follows, we simply denote $\Scal^p (L^2(\partial\Omega))$ by $\Scal^p$.

\begin{theo}
    \label{theo_singular_value_decay}
    If $0<\alpha\le 1$ and $\partial\Omega$ is $C^{1, \alpha}$, then we have $\Kcal^*\in \Scal^p$ for all
    $p>d/\Ga$.
\end{theo}

\begin{proof}
    Let $n$ be the smallest integer larger than $d/(2\alpha)$. We infer from Lemma \ref{lemm_kernel_sobolev} and Theorem \ref{theo_Delgado_Ruzhansky} that $\Lcal_n \in \Scal^p (L^2(\partial\Omega))$ for all
    \[
        p>\frac{2d}{d+2(2n\alpha-d)/2}=\frac{d}{n\alpha}.
    \]
    Since $\Kcal\Kcal^*$ is a nonnegative self-adjoint operator on $L^2(\partial\Omega)$ and $\Lcal_n^*\Lcal_n=(\Kcal\Kcal^*)^n$, we have
    \[
        s_j (\Lcal_n)=\lambda_j ((\Lcal_n^*\Lcal_n)^{1/2})=\lambda_j((\Kcal\Kcal^*)^{1/2})^n =s_j (\Kcal^*)^n.
    \]
    Thus, if $p>d/\alpha$, then we have
    \[
        \sum_{j=1}^\infty s_j(\Kcal^*)^p=\sum_{j=1}^\infty s_j (\Lcal_n)^{p/n}<\infty.
    \]
    Thus, $\Kcal^*\in \Scal^p$.
\end{proof}

Theorem \ref{theo_eigenvalue_decay} follows immediately.

\begin{proof}[Proof of Theorem \ref{theo_eigenvalue_decay}]
    Since $\partial\Omega$ is $C^{1,\alpha}$ for some $\alpha>0$, a non-zero eigenvalue of $\Kcal^*$ on $H^{-1/2}(\partial\Omega)$ is also an eigenvalue on $L^2(\partial\Omega)$, and vice versa (see Corollary \ref{coro_eigenfunction_smooth}). We apply Theorem \ref{theo_singular_value_decay} and obtain $\Kcal^* \in \Scal^p$ for all $p>d/\alpha$. Hence, \eqref{eq_eigenvalue_decay} follows by Lemma \ref{lemm_schatten_decay}.
\end{proof}

\section{Proof of Theorem \ref{theo_eigenvalue_decay_2d_smoother}}\label{theorem1.2}

\subsection{Estimate of tangential derivatives}

We begin with the estimate of the kernel analogous to Lemma \ref{lemm_L1_kernel_estimate}. We note that the function $y\mapsto L_1(x, y)$ is $C^{k, \alpha}$ is $k$-times continuously differentiable in $\partial\Omega \setminus \{x\}$ for fixed $x\in \partial\Omega$ and the $k$th derivative $\partial_{\tau_y}^k L_1(x, y)$ is continuous in the space
\[
    (\partial\Omega\times \partial\Omega)\setminus \{ (x, x) \mid x\in \partial\Omega\}.
\]
Here and throughout this section $\partial_{\tau}$ denotes the tangential derivative along $\partial\Omega$.

\begin{lemm}
    \label{lemm_tangential_derivative_ckb}
    Let $\Omega$ be a bounded domain in $\Rbb^{2}$ with $C^{k, \alpha}$ boundary $\partial\Omega$ for some positive integer $k$ and $\alpha \in (0,1]$. There exists a constant $C>0$ such that the inequality
    \begin{equation}
        |\partial_{\tau_y}^l L_1(x, y)|\leq \frac{C}{|x-y|^{2-k+l-\alpha}} \label{eq_tangential_derivative_k-1_ckb}
    \end{equation}
    holds for all $x, y\in \partial\Omega$ with $x\neq y$ and $l=0, 1, \ldots, k$. In particular, we have
    \begin{equation}
        \label{eq_tangential_derivative_k_ckb}
        |\partial_{\tau_y}^{k-1} L_1(x, y)-\partial_{\tau_y}^{k-1} L_1(x, y^\prime)|\leq \frac{C|y-y^\prime|}{|x-y|^{2-\alpha}}
    \end{equation}
    if $2|y-y^\prime|<|x-y|$.
\end{lemm}

For the proof of Lemma \ref{lemm_tangential_derivative_ckb}, we need some preparation. Let $\gamma: \Rbb \to \partial\Omega$, which has the period $1$, be a $C^{k, \alpha}$ parametrization of $\partial\Omega$ such that $\gamma'(t) \neq 0$ for any $t$. By the Taylor theorem, we can expand $\gamma (t)$ as
\begin{equation}
    \label{eq_gamma_taylor}
    \gamma (s)-\gamma (t)=\sum_{j=1}^k \frac{\gamma^{(j)}(t)}{j!}(s-t)^j +R_k (t, s)
\end{equation}
where
\begin{equation}
    \label{eq_gamma_taylor_remainder}
    R_k (t, s):=\frac{1}{(k-1)!}\int_t^s (s-\sigma)^{k-1}(\gamma^{(k)}(\sigma)-\gamma^{(k)}(t))\, \df \sigma.
\end{equation}

\begin{lemm}
    \label{lemm_estimate_remainder}
    For integers $l=0, 1, 2, \ldots, k$ and $m\geq 0$, we have
    \[
        \frac{\partial^l}{\partial s^l}\left(\frac{R_k(t, s)}{(t-s)^m}\right)
        =O(|t-s|^{k-m-l+\alpha})
    \]
    as $|t-s|\to 0$.
\end{lemm}

\begin{proof}
    We first consider the case of $m=0$. We have
    \begin{align*}
        &\frac{\partial^l}{\partial s^l}R_k(t, s) \\
        &=
        \begin{dcases}
            \frac{1}{(k-l-1)!}\int_t^s (s-\sigma)^{k-l-1}(\gamma^{(k)}(\sigma)-\gamma^{(k)}(t))\, \df \sigma & \text{if } 0\leq l\leq k-1, \\
            \gamma^{(k)}(s)-\gamma^{(k)}(t) & \text{if } l=k.
        \end{dcases}
    \end{align*}
    Hence we obtain
    \begin{align*}
        \left|\frac{\partial^l}{\partial s^l}R_k(t, s)\right|
        &\lesssim \begin{dcases}
            \left|\int_t^s (s-\sigma)^{k-l-1}|\sigma-t|^\alpha \, \df \sigma\right| & \text{if } 0\leq l \leq k-1, \\
            |s-t|^\alpha & \text{if } l=k
        \end{dcases} \\
        &\lesssim |s-t|^{k-l+\alpha}.
    \end{align*}

    Next we consider a general $m\geq 1$. By the Leibnitz rule, we have
    \begin{align*}
        \frac{\partial^l}{\partial s^l}\left(\frac{R_k(t, s)}{(t-s)^m}\right)
        &=\frac{1}{(k-1)!(m-1)!}\sum_{r=0}^l \begin{pmatrix}
            l \\
            r
        \end{pmatrix}
        \frac{(m+l-r-1)!}{(t-s)^{m+l-r}}\frac{\partial^r}{\partial s^r}R_k (t, s).
    \end{align*}
    Using estimates for the case $m=0$, we obtain
    \[
        \left|\frac{\partial^l}{\partial s^l}\left(\frac{R_k(t, s)}{(t-s)^m}\right)\right|
        \lesssim \sum_{r=0}^l \frac{1}{|t-s|^{m+l-r}}|s-t|^{k-r+\alpha}
        \lesssim |s-t|^{k-m-l+\alpha}
    \]
    as desired.
\end{proof}

\begin{lemm}
    \label{lemm_inverse_difference_estimate}
    For $l=0, 1, \ldots, k$, we have
    \[
        \frac{\partial^l}{\partial s^l}\frac{(t-s)^2}{|\gamma (t)-\gamma (s)|^2}=
        \begin{cases}
            O(1) & \text{if } 0\leq l\leq k-1, \\
            O(|t-s|^{-1+\alpha}) & \text{if } l=k
        \end{cases}
    \]
    as $|t-s|\to 0$.
\end{lemm}

\begin{proof}
    We begin with the estimate of
    \[
        \frac{\partial^l}{\partial s^l}\frac{|\gamma (t)-\gamma (s)|^2}{(t-s)^2}
    \]
    for $0\leq l\leq k$. By the Leibnitz rule, we have
    \[
        \frac{\partial^l}{\partial s^l}\frac{|\gamma (t)-\gamma (s)|^2}{(t-s)^2}
        =\sum_{r=0}^l \begin{pmatrix}
            l \\
            r
        \end{pmatrix}
        \frac{\partial^r}{\partial s^r} \frac{\gamma (t)-\gamma (s)}{t-s}\cdot \frac{\partial^{l-r}}{\partial s^{l-r}} \frac{\gamma (t)-\gamma (s)}{t-s}
    \]
    and thus
    \begin{equation}
        \label{eq_leibnitz_estimate}
        \left|\frac{\partial^l}{\partial s^l}\frac{|\gamma (t)-\gamma (s)|^2}{(t-s)^2}\right|
        \lesssim \sum_{r=0}^l \left|
        \frac{\partial^r}{\partial s^r} \frac{\gamma (t)-\gamma (s)}{t-s}\right|\left|\frac{\partial^{l-r}}{\partial s^{l-r}} \frac{\gamma (t)-\gamma (s)}{t-s}\right|.
    \end{equation}
    By the Taylor expansion \eqref{eq_gamma_taylor} and Lemma \ref{lemm_estimate_remainder}, we have
    \begin{align*}
        \frac{\partial^r}{\partial s^r}\frac{\gamma (t)-\gamma (s)}{t-s}
        &=\frac{\partial^r}{\partial s^r}\left(\sum_{j=1}^k \frac{\gamma^{(j)}(t)}{j!}(s-t)^{j-1} +\frac{R_k (t, s)}{s-t}\right) \\
        &=O(1)+O(|t-s|^{k-1-r+\alpha}) \\
        &=\begin{cases}
            O(1) & \text{if } r\leq k-1, \\
            O(|t-s|^{-1+\alpha}) & \text{if } r=k.
        \end{cases}
    \end{align*}
    It thus follows from \eqref{eq_leibnitz_estimate} that
    \begin{equation}
        \label{eq_leibnitz_estimate_final}
        \left|\frac{\partial^l}{\partial s^l}\frac{|\gamma (t)-\gamma (s)|^2}{(t-s)^2}\right|
        \lesssim
        \begin{cases}
            1 & \text{if } 0\leq l\leq k-1, \\
            |t-s|^{-1+\alpha} & \text{if } l=k.
        \end{cases}
    \end{equation}

    Now we set
    \[
        F(t, s):=\frac{|\gamma (t)-\gamma (s)|^2}{(t-s)^2}.
    \]
    and estimate the derivatives of $1/F(t, s)$. We begin with the formula
    \begin{equation}
        \label{eq_faa_di_bruno}
        \frac{\partial^l}{\partial s^l}\frac{(t-s)^2}{|\gamma (t)-\gamma (s)|^2}
        =\sum_{m=1}^l  \frac{(-1)^m m!}{F(t, s)^{m+1}}\sum_{\substack{i_1+\cdots +i_m=l \\ i_1, \ldots, i_m\geq 1}} \prod_{r=1}^m \frac{\partial^{i_r}}{\partial s^{i_r}}F(t, s).
    \end{equation}
    Since $F(t, s)$ tens to $|\gamma' (t)|^2$ as $s\to t$ and $\gamma'$ is nowhere vanishing, we can estimate \eqref{eq_faa_di_bruno} as
    \begin{equation}
        \label{eq_faa_di_bruno_estimate}
        \begin{aligned}
            \left|\frac{\partial^l}{\partial s^l}\frac{(t-s)^2}{|\gamma (t)-\gamma (s)|^2}\right|
            &\lesssim \sum_{m=1}^l  \sum_{\substack{i_1+\cdots +i_m=l \\ i_1, \ldots, i_m\geq 1}} \prod_{r=1}^m \left|\frac{\partial^{i_r}}{\partial s^{i_r}}F(t, s)\right| \\
            &\lesssim \begin{cases}
                1 & \text{if } 0\leq l \leq k-1, \\
                |t-s|^{-1+\alpha} & \text{if } l=k
            \end{cases}
        \end{aligned}
    \end{equation}
    by using \eqref{eq_leibnitz_estimate_final}. This is the desired estimate.
\end{proof}

Now we are ready to prove Lemma \ref{lemm_tangential_derivative_ckb}.

\begin{proof}[Proof of Lemma \ref{lemm_tangential_derivative_ckb}]
    It is enough to prove
    \[
        \frac{\partial^l}{\partial s^l}L_1 (\gamma (t), \gamma (s))=O(|t-s|^{-2+k-l+\alpha})
    \]
    as $|t-s|\to 0$. Since $\gamma' (t)\cdot \nv_{\gamma (t)}=0$, it follows from the Taylor expansion \eqref{eq_gamma_taylor} that
    \begin{align*}
        &L_1(\gamma (t), \gamma (s)) \\
        &=-\frac{(t-s)^2}{2\pi |\gamma (t)-\gamma (s)|^2}\left(\sum_{j=0}^{k-2} \frac{\gamma^{(j+2)}(t)\cdot \nv_{\gamma (t)}}{(j+2)!}(s-t)^j+\frac{R_k (t, s)\cdot \nv_{\gamma (t)}}{(t-s)^2}\right).
    \end{align*}
    We apply the Leibnitz rule and obtain
    \begin{align*}
        &\frac{\partial^l}{\partial s^l}L_1(\gamma (t), \gamma (s)) \\
        &=-\frac{1}{2\pi}\sum_{r=0}^l \begin{pmatrix}
            l \\
            r
        \end{pmatrix}
        \frac{\partial^{l-r}}{\partial s^{l-r}}\frac{(t-s)^2}{|\gamma (t)-\gamma (s)|^2} \\
        &\quad \times \frac{\partial^r}{\partial s^r}\left(\sum_{j=0}^{k-2} \frac{\gamma^{(j+2)}(t)\cdot \nv_{\gamma (t)}}{(j+2)!}(s-t)^j+\frac{R_k (t, s)\cdot \nv_{\gamma (t)}}{(t-s)^2}\right).
    \end{align*}
    Now we employ Lemma \ref{lemm_estimate_remainder} and Lemma \ref{lemm_inverse_difference_estimate} and obtain the estimate
    \begin{align*}
        &\left|\frac{\partial^l}{\partial s^l}L_1(\gamma (t), \gamma (s))\right| \\
        &\lesssim \sum_{r=0}^l
        \left|\frac{\partial^{l-r}}{\partial s^{l-r}}\frac{(t-s)^2}{|\gamma (t)-\gamma (s)|^2}\right| \\
        &\quad \times \left|\frac{\partial^r}{\partial s^r}\left(\sum_{j=0}^{k-2} \frac{\gamma^{(j+2)}(t)\cdot \nv_{\gamma (t)}}{(j+2)!}(s-t)^j+\frac{R_k (t, s)\cdot \nv_{\gamma (t)}}{(t-s)^2}\right)\right| \\
        &\lesssim \begin{cases}
            |t-s|^{k-l-2+\alpha} & \text{if } 0\leq l\leq k-1, \\
            (|t-s|^{-1+\alpha}+1+|t-s|^{-2+\alpha}) \lesssim |t-s|^{-2+\alpha} & \text{if } l=k.
        \end{cases}
    \end{align*}
    So, we have \eqref{lemm_tangential_derivative_ckb}. The estimate \eqref{eq_tangential_derivative_k_ckb} is an immediate consequence of \eqref{lemm_tangential_derivative_ckb} for $l=k$.
\end{proof}

\subsection{Sobolev condition and decay rates in 2D}

\begin{lemm}
    \label{lemm_sobolev_condition_ckb}
    If $\Omega$ be a bounded domain in $\Rbb^{2}$ with $C^{k, \alpha}$ boundary $\partial\Omega$ for some positive integer $k$ and $\alpha \in (0,1]$, then $L_1\in H^{0, \nu}(\partial\Omega \times \partial\Omega)$ for all $\nu\in [0, k+\alpha-3/2)$.
\end{lemm}

\begin{proof}
    Suppose first that $0<\alpha \le 1/2$. If $p\in [1, 1/(1-\alpha))$, then
    \[
        \sup_{x\in \partial\Omega} \|\partial_{\tau_y}^{k-1} L_1(x, y)\|_{L^p (\partial\Omega)}^p
        \leq C\sup_{x\in \partial\Omega} \int_{\partial\Omega} \frac{\df \sigma (y)}{|x-y|^{p(1-\alpha)}}<\infty.
    \]
    We also have
    \[
        \sup_{x\in \partial\Omega} \|\partial_{\tau_y}^{k-1}L_1(x, y)\|_{L^p (\partial\Omega)}^p <\infty.
    \]
    Thus we have
    \begin{equation}
        \label{eq_nps_kernel_sobolev_p_ckb}
        L_1\in L^\infty_x (\partial\Omega, W^{k-1, p}_y (\partial\Omega)), \quad \forall p\in [1, 1/(1-\alpha)).
    \end{equation}
    By the Sobolev embedding theorem, we have the continuous embedding $W^{k-1, p}(\partial\Omega) \hookrightarrow H^\nu (\partial\Omega)$ for $\nu =k-1/2-1/p$. Hence we have $L_1\in L^\infty_x (\partial\Omega, H^\nu_y (\partial\Omega))$ for $\nu = k-1/2-1/p$. Here we remark that the condition $\nu\leq k-1$ is automatically satisfied by the assumption $0<\alpha\leq 1/2$. Since $p\in [1, 1/(1-\alpha))$ is arbitrary, we have $L_1\in L^\infty_x (\partial\Omega, H^\nu_y (\partial\Omega))$ for all $0\leq \nu <k+\alpha-3/2$. By the continuous embedding $L^\infty (\partial\Omega) \hookrightarrow L^2(\partial\Omega)$, we obtain $L_1\in L^2_x(\partial\Omega, H^\nu_y (\partial\Omega))=H^{0, \nu} (\partial\Omega \times \partial\Omega)$.

    Next we look into the case when $1/2<\alpha<1$. We prove that $\partial_{\tau_y}^{k-1} L_1(x, y)$ belong to $H^{0, \nu}(\partial\Omega \times \partial\Omega)$ for all $\nu\in [0, \alpha-1/2)$.
    Since $L_1\in L^\infty_x (\partial\Omega, H^{k-1}_y (\partial\Omega))$ by \eqref{eq_tangential_derivative_k-1_ckb}, it suffices to prove (by \eqref{eq_sobolev}) that
    \begin{equation}
        \label{eq_slobodeckij_2_ckb}
        \sup_{x\in \partial\Omega} \int_{\partial\Omega\times \partial\Omega} \frac{|\partial_{\tau_y}^{k-1} L_1(x, y)-\partial_{\tau_y}^{k-1} L_1(x, y^\prime)|^2}{|y-y^\prime|^{2\nu+1}}\, \df \sigma (y)\df\sigma (y^\prime)<\infty
    \end{equation}
    for all $\nu\in [0, \alpha-1/2)$.  This can be proved in the exactly same way as \eqref{eq_kernel_sobolev_seminorm}.
\end{proof}

\begin{proof}[Proof of Theorem \ref{theo_eigenvalue_decay_2d_smoother}]
    Thanks to Theorem \ref{theo_Delgado_Ruzhansky}, Lemma \ref{lemm_sobolev_condition_ckb} implies that $\Kcal^*\in \Scal^p$ for all $p>1/(k+\alpha-1)$. Thus we obtain \eqref{700} by Lemma \ref{lemm_schatten_decay}.
\end{proof}

\appendix

\section{Regularity properties of the NP operator}\label{appendix}

The purpose of this section is to give a proof of the fact that if $\partial\Omega$ is $C^{1, \alpha}$ for some $\alpha>0$, then the NP operator $\Kcal^*$ is compact operator on $H^{-1/2}(\partial\Omega)$. We also prove that every eigenfunction of $\Kcal^*$ in $H^{-1/2}(\partial\Omega)$ corresponding to a non-zero eigenvalue belongs to $L^2 (\partial\Omega)$.

Let $\Kcal$ be an integral operator 
\begin{equation}
    \label{eq_np_ns}
    \Kcal [f](x):=-\frac{1}{\omega_{d+1}}\mathrm{p.v.}\int_{\partial\Omega} \frac{(x-y)\cdot \nv_y}{|x-y|^{d+1}}f(y)\, \df \sigma (y) \quad (x\in \partial\Omega), 
\end{equation}
which is also called Neumann-Poincar\'e operator and satisfies the duality 
\[
    \jbk{f, \Kcal^*[g]}_{L^2(\partial\Omega)}=\jbk{\Kcal[f], g}_{L^2(\partial\Omega)}
\]
for all $f, g\in L^2(\partial\Omega)$. We begin with the following theorem. Let $\Omega\subset \Rbb^{d+1}$ be a bounded domain.

\begin{theo}\label{theo_NP_bounded_Sobolev}
    Let $0<\alpha<1$ and $\partial\Omega$ be $C^{1, \alpha}$. Then $\Kcal$ is a bounded operator from $H^s(\partial\Omega)$ into $H^{s+\alpha} (\partial\Omega)$ if $0 <s<(1-\alpha)/2$ and from $L^2(\partial\Omega)$ into $H^{\alpha-\varepsilon} (\partial\Omega)$ for any $\varepsilon >0$.
\end{theo}

\begin{proof}
    Let $f \in H^s (\partial\Omega)$ for some $s \in [0, (1-\alpha)/2)$. Let $\varepsilon$ be any number such that $0< \varepsilon <\alpha/2$ and define the number $t$ by
    \[
    t=
    \begin{cases}
    s \quad &\text{if } s \in (0, (1-\alpha)/2), \\
    -\varepsilon  &\text{if } s=0.
    \end{cases}
    \]
    We prove that 
    \[
    \| \Kcal[f] \|_{H^{t+\alpha}(\partial\Omega)} \lesssim  \| f \|_{H^s(\partial\Omega)}.
    \]
    Since $\| \Kcal[f] \|_{L^2(\partial\Omega)} \lesssim  \| f \|_{H^s(\partial\Omega)}$, it suffice to prove
    \begin{equation}\label{10000}
        I:= \int_{\partial\Omega\times \partial\Omega} \frac{|\Kcal[f] (y)-\Kcal[f] (x^\prime)|^2}{|x-x^\prime|^{2(t+\alpha)+d}}\, \df \sigma (x)\df \sigma (x^\prime) \lesssim \| f \|_{H^s(\partial\Omega)}
    \end{equation}
    by \eqref{eq_sobolev}. 
    
    For ease of notation let $K(x,y)$ be the integral kernel of $\Kcal$, namely,
    \[
    K(x,y)= \frac{1}{\omega_{d+1}} \frac{(x-y)\cdot \nv_y}{|x-y|^{d+1}}.
    \]
    Since $\Kcal[1]=1/2$, we have
    \begin{align*}
        \Kcal[f](x)-\Kcal[f](x^\prime)
        &=\int_{\partial\Omega}\left(K(x,y)- K(x^\prime,y) \right)(f (y)-f (x))\, \df \sigma (y) \\
        &=\int_{2|x-x^\prime|<|x-y|} + \int_{2|x-x^\prime|\geq |x-y|}=: J_1+J_2,
    \end{align*}
    so that
    \[
        I \lesssim \sum_{j=1}^2 \int_{\partial\Omega\times \partial\Omega} \frac{|J_j|^2}{|x-y|^{2(t+\alpha)+d}}\, \df \sigma (x)\df \sigma (y) :=I_1+I_2.
    \]

    Following the same lines of the proof of Lemma \ref{lemm_L1_kernel_estimate}, we have
    \[
        \left| K(x,y)- K(x^\prime,y) \right|\lesssim \frac{|x-x^\prime|}{|x-y|^{d+1-\alpha}}
    \]
    for all $x, x^\prime, y\in \partial\Omega$ with $2|x-x^\prime|<|x-y|$. If we set $A_1:= \{y \mid 2|x-x^\prime|<|x-y| \}$, then we have
    \[
    |J_1| \lesssim \int_{A_1}\frac{|x-x^\prime|}{|x-y|^{d+1-\alpha}}|f (y)-f (x)|\, \df \sigma (y).
    \]
    We apply the Cauchy-Schwarz inequality to have
    \begin{align*}
        |J_1|^2 &\lesssim |x-x^\prime|^2\left( \int_{A_1}\frac{\df \sigma (y)}{|x-y|^{d+1-\alpha}}\right) \left( \int_{A_1}\frac{|f (y)-f(x)|^2\, \df \sigma (y)}{|x-y|^{d+1-\alpha}}\right) \\
        &\lesssim |x-x^\prime|^{1+\alpha}\int_{A_1}\frac{|f (y)-f (x)|^2}{|x-y|^{d+1-\alpha}}\, \df \sigma (y).
    \end{align*}
    Since $t<(1-\alpha)/2$, we have
    \begin{align*}
        I_1& \lesssim\int_{\partial\Omega\times \partial\Omega}\df \sigma (x)\df \sigma (y)\, \frac{|f (y)-f (x)|^2}{|x-y|^{d-\alpha}}\int_{2|x-x^\prime|<|x-y|}\frac{\df \sigma (x^\prime)}{|x-x^\prime|^{2t+\alpha+d-1}} \\
        &\lesssim \int_{\partial\Omega\times \partial\Omega} \frac{|f (y)-f (x)|^2}{|x-y|^{d+2t}}\, \df \sigma (x)\df \sigma (y) \\
        &\leq \|f\|_{H^{s}(\partial\Omega)}^2.
    \end{align*}

    Set $A_2:= \{y \mid 2|x-x^\prime| \ge |x-y| \}$. For $J_2$, we have
    \begin{align*}
        |J_2|&\lesssim \int_{A_2}\left(|K(x,y)|+|K(x^\prime,y)| \right)|f (y)-f (x)|\, \df \sigma (y) \\
        &\lesssim \int_{A_2}\left(\frac{1}{|x-y|^{d-\alpha}}+\frac{1}{|x^\prime-y|^{d-\alpha}} \right)|f (y)-f (x)|\, \df \sigma (y).
    \end{align*}
    We then apply the Cauchy-Schwarz inequality to have
    \begin{align*}
        |J_2|^2&\lesssim \left(\int_{A_2}\frac{\df \sigma (y)}{|x-y|^{d-\alpha}}\right)\left(\int_{A_2}\frac{|f (y)-f (x)|^2\, \df \sigma (y)}{|x-y|^{d-\alpha}}\right) \\
        &\quad + \left(\int_{A_2}\frac{\df \sigma (y)}{|x^\prime-y|^{d-\alpha}}\right) \left(\int_{A_2|}\frac{|f (y)-f (x^\prime)|^2\, \df \sigma (y)}{|x^\prime-y|^{d-\alpha}} \right) \\
        &\quad + \left(\int_{A_2}\frac{\df \sigma (y)}{|x^\prime-y|^{d-\alpha}}\right)^2  |f (x)-f (x^\prime)|^2.
    \end{align*}
    If $y \in A_2$, then $|x^\prime-y|\leq |x-x^\prime|+|x-y|\leq 3|x-x^\prime|$. Thus we have
    \begin{align*}
        |J_2|^2 &\lesssim |x-x^\prime|^\alpha\int_{2|x-x^\prime|\geq |x-y|}\frac{|f (y)-f (x)|^2\, \df \sigma (y)}{|x-y|^{d-\alpha}} \\
        &\quad +|x-x^\prime|^\alpha\int_{3|x-x^\prime|\geq |x^\prime-y|}\frac{|f (y)-f (x^\prime)|^2\, \df \sigma (y)}{|x^\prime-y|^{d-\alpha}} \\
        &\quad +|x-x^\prime|^{2\alpha} |f (x)-f (x^\prime)|^2 =: \sum_{j=1}^3 J_{2,j},
    \end{align*}
    and hence
    \[
    I_2 \lesssim \sum_{j=1}^3 I_{2,j},
    \]
    where the definition $I_{2,j}$ is obvious. For $I_{2,1}$, we have
    \begin{align*}
        I_{2,1} & \lesssim \int_{\partial\Omega\times \partial\Omega}\df \sigma (x)\df \sigma (y)\, \frac{|f (y)-f (x)|^2}{|x-y|^{d-\alpha}}\int_{2|x-x^\prime|\geq |x-y|}\frac{\df \sigma (x^\prime)}{|x-x^\prime|^{2t+\alpha+d}} \\
        &\lesssim \int_{\partial\Omega\times \partial\Omega}\frac{|f (y)-f (x)|^2}{|x-y|^{d+2t}}\, \df \sigma (x)\df \sigma (y)
        \lesssim \| f \|_{H^{s}(\partial\Omega)}^2,
    \end{align*}
    where the second inequality holds since $t \ge -\varepsilon > -\alpha/2$. One can estimate $I_{2,2}$ similarly to have
    \[
    I_{2,2} \lesssim \| f \|_{H^{s}(\partial\Omega)}^2.
    \]
    We also have
    \[
    I_{2,3} \lesssim \int_{\partial\Omega\times \partial\Omega}\frac{|f (y)-f (x)|^2}{|x-y|^{d+2t}}\, \df \sigma (x)\df \sigma (y)
        \lesssim \| f \|_{H^{s}(\partial\Omega)}^2.
    \]
    So we have \eqref{10000} and the proof is complete.
\end{proof}

\begin{coro}
    \label{coro_NP_compact_Sobolev}
    If $\partial\Omega$ is $C^{1, \alpha}$ for some $\alpha>0$, then the NP operators $\Kcal$ and $\Kcal^*$ are compact operators on $H^{1/2}(\partial\Omega)$ and $H^{-1/2}(\partial\Omega)$, respectively.
\end{coro}

\begin{proof}
    We take $s\in (\max\{0, 1/2-\alpha\}, (1-\alpha)/2)$. Then we apply Theorem \ref{theo_NP_bounded_Sobolev} to obtain the boundedness of $\Kcal: H^s(\partial\Omega)\to H^{s+\alpha}(\partial\Omega)$. Since $s<(1-\alpha)/2<1/2$, the inclusion $H^{1/2}(\partial\Omega)\hookrightarrow H^s(\partial\Omega)$ is compact by the Rellich-Kondrachov embedding theorem. Since $s+\alpha> 1/2$, we have the continuous (actually, compact) embedding $H^{s+\alpha}(\partial\Omega)\hookrightarrow H^{1/2}(\partial\Omega)$. Thus, the composition
    \[
        H^{1/2}(\partial\Omega) \longhookrightarrow H^s(\partial\Omega) \overset{\Kcal}{\longrightarrow} H^{s+\alpha}(\partial\Omega)\longhookrightarrow H^{1/2}(\partial\Omega)
    \]
    gives a compact operator on $H^{1/2}(\partial\Omega)$.

    By the duality, $\Kcal^*: H^{-1/2}(\partial\Omega)\to H^{-1/2}(\partial\Omega)$ is compact.
\end{proof}

\begin{coro}\label{coro_eigenfunction_smooth}
    Suppose that $\partial\Omega$ is $C^{1, \alpha}$ for some $\alpha>0$.
    \begin{enumerate}[label=(\roman*)]
        \item \label{enum_eigenfunction_smooth_np} Every eigenfunction of $\Kcal$ in $L^2(\partial\Omega)$ corresponding to a non-zero eigenvalue belongs to $H^{1/2} (\partial\Omega)$.
        \item \label{enum_eigenfunction_smooth_nps}Every eigenfunction of $\Kcal^*$ in $H^{-1/2}(\partial\Omega)$ corresponding to a non-zero eigenvalue belongs to $L^2 (\partial\Omega)$.
    \end{enumerate}
\end{coro}

\begin{proof}
    \ref{enum_eigenfunction_smooth_np} Let $f\in L^2(\partial\Omega)$ satisfy $\Kcal[f]=\lambda f$ with $\lambda\neq 0$. Then, by Theorem \ref{theo_NP_bounded_Sobolev}, we have
    \[
        f=\lambda^{-1}\Kcal [f]\in H^{\alpha/2}(\partial\Omega).
    \]
    We employ \ref{theo_NP_bounded_Sobolev} again and obtain
    \[
        f=\lambda^{-1}\Kcal [f]\in H^{\min \{1/2, 3\alpha/2\}}(\partial\Omega).
    \]
    We iterate this procedure and obtain $f\in H^{1/2}(\partial\Omega)$.

    \freespace\noindent\ref{enum_eigenfunction_smooth_nps} Theorem \ref{theo_NP_bounded_Sobolev} implies by duality that the operator $\Kcal^*$ is bounded from $H^{-s-\alpha}(\partial\Omega)$ into $H^{-s}(\partial\Omega)$ if $0<s<(1-\alpha)/2$ and from $H^{-\alpha+\varepsilon}(\partial\Omega)$ into $L^2(\partial\Omega)$ for any $\varepsilon >0$. Thus the same argument as in \ref{enum_eigenfunction_smooth_np} shows the assertion.
\end{proof}


\end{document}